\documentclass[a4paper,fleqn]{cas-dc}

\usepackage[numbers]{natbib}

\usepackage[utf8]{inputenc}
\usepackage{amsmath}
\usepackage{amsfonts}
\usepackage{amssymb}
\usepackage{geometry}
\usepackage{framed}
\usepackage{enumerate}
\usepackage{hyperref}
\hypersetup{
	colorlinks=true,        
	linkcolor=blue,         
	citecolor=blue,         
	urlcolor=blue           
}
\usepackage[ruled,vlined]{algorithm2e}
\usepackage{algpseudocode}
\usepackage{caption}
\usepackage{subcaption}

\usepackage{amsthm}
\usepackage{thmtools}

\newtheorem{theorem}{Theorem}[section]
\newtheorem{proposition}{Proposition}[section]

\newtheorem{definition}{Definition}[section]

\newtheorem{lemma}{Lemma}[section]

\newcommand{\range}{\operatorname{range}}
\DeclareMathOperator*{\argmin}{arg\,min}
\newcommand{\Hilbert}{\mathcal{H}}

\newcommand{\Fix}{\operatorname{Fix}}

\begin{document}
\title[mode=title]{Gearhart--Koshy Acceleration for Affine Subspaces}
\shorttitle{Gearhart--Koshy Acceleration for Affine Subspaces}
\author[1]{\color{black}Matthew K.\ Tam}[orcid=0000-0002-3654-6553]
\shortauthors{M.\ K.\ Tam}
\address[1]{School of Mathematics \& Statistics, The University of Melbourne, Parkville VIC 3010, Australia.}
\ead{matthew.tam@unimelb.edu.au}                         

\begin{abstract}
The method of cyclic projections finds nearest points in the intersection of finitely many affine subspaces. To accelerate convergence, Gearhart \& Koshy proposed a modification which, in each iteration, performs an exact line search based on minimising the distance to the solution. When the subspaces are linear, the procedure can be made explicit using feasibility of the zero vector. This work studies an alternative approach which does not rely on this fact, thus providing an efficient implementation in the affine setting.
\end{abstract}

\smallskip

\begin{keywords}
cyclic projections \sep acceleration schemes \sep linear systems
\end{keywords}

\maketitle

\section{Introduction}
Our setting is a real Hilbert space $\Hilbert$ equipped with inner-product $\langle\cdot,\cdot\rangle$ and induced norm $\|\cdot\|$. Consider closed affine subspaces $M_1,\dots,M_n\subseteq\Hilbert$, and suppose
\begin{equation*}
M := \bigcap_{i=1}^nM_i\neq\emptyset.
\end{equation*}
Given $x_0\in\Hilbert$, we study the \emph{best approximation problem}
  \begin{equation}\label{eq:bap}
   \min_{x\in\Hilbert} \|x-x_0\|^2\text{~subject to~}x\in M.
  \end{equation}
In this work, our focus is the case in which the nearest point projectors onto the individual spaces, $M_1,\dots,M_n$, are accessible. Recall that the \emph{projector} onto $M_i$ is the operator $P_{M_i}\colon\Hilbert\to M_i$ given by
\begin{equation*}
P_{M_i}(x) := \argmin_{z\in M_i}\|x-z\|.
\end{equation*}
The \emph{method of cyclic projections} is an iterative procedure for solving \eqref{eq:bap} (\emph{i.e.,} for computing $P_M(x_0)$) by using only the individual projection operators $P_{M_1},\dots,\allowbreak P_{M_n}$. Although originally studied when $M_1,\dots,M_n$ are linear subspaces \cite{N1950,H1962}, the following affine variant readily follows from translation properties of the projector. 
\begin{theorem}[The method of cyclic projections]\label{th:mcp}
Let $M_1,\dots,M_n$ be closed affine subspaces of $\Hilbert$ with $M=\cap_{i=1}^nM_i\neq \emptyset$. Then, for each $x_0\in \Hilbert$, 
\begin{equation*}
  \lim_{k\to\infty}(P_{M_n}P_{M_{n-1}}\dots P_{M_1})^k(x_0)=P_M(x_0). 
\end{equation*}
\end{theorem}  
The convergence rate of the sequence in Theorem~\ref{th:mcp} can be related to the angle between the subspaces. Recall that the \emph{(Friederichs) angle} between two closed subspaces $A$ and $B$ is the angle in $[0,\pi/2]$ whose cosine is given by
\begin{equation*}
 c(A,B) := \sup\left\{|\langle a,b\rangle|:\begin{array}{c} a\in A\cap (A\cap B)^\perp \\ b\in B\cap (A\cap B)^\perp\\ \|a\|=\|b\|=1\end{array}\right\}. 
\end{equation*}
The following result provides a bound on the convergence rate based on this quantity.
\begin{theorem}[{\cite[Corollary~9.34]{D2012}}]\label{th:affine rate}
Let $M_1,\dots,M_n$ be closed affine subspaces of $\Hilbert$ with $M=\cap_{i=1}^nM_i\neq \emptyset$. For $i\in\{1,\dots,n\}$, let $M_i'$ denote the linear subspace parallel to $M_i$. Then, for each $x_0\in \Hilbert$, 
\begin{multline}\label{eq:linear bound}
 \|(P_{M_n}P_{M_{n-1}}\dots P_{M_1})^k(x_0)-P_M(x_0)\| \\
 \leq c^k \|x_0-P_M(x_0)\|, 
\end{multline}
where the constant $c\in[0,1]$ is given by
\begin{multline}\label{eq:c}
c:=\left(1-\prod_{i=1}^{n-1}(1-c_i^2)\right)^{1/2}\\\text{with~~} c_i:=c\left(M_i',\bigcap_{j=i+1}^nM_j'\right). 
\end{multline}
\end{theorem}
When $c<1$, Theorem~\ref{th:affine rate} establishes $R$-linear convergence of the method of cyclic projections. This is easily seen to be the case, for instance, when $c_i<1$ for all $i\in \{1,\dots,n\}$. In the setting with $n=2$, this characterisation can be further refined: $c<1$ if and only if $M_1+M_2$ is closed (which always holds in finite dimensions) in which case convergence is linear, else $c=1$ and the rate of convergence is arbitrarily slow \cite{BBL1997,BDH2009}.

Let $Q\colon\Hilbert\to\Hilbert$ denote the cyclic projections operator given by $Q:=P_{M_n}\dots P_{M_1}$. In an attempt to accelerate the method of cyclic projections, Gearhart \& Koshy \cite{GK1989} proposed the following scheme which iterates by performing an exact line search to choose to nearest point to $P_M(x_0)$ in the affine span of $\{x_k,Q(x_k)\}$. 
\begin{algorithm}[ht]
\caption{Gearhart--Koshy (1989) acceleration for \eqref{eq:bap}.}\label{a:gk}
\textbf{Initialisation.}~An initial point $x_0\in\Hilbert$.\\
\For{$k=0,1,2,\dots$}{
1.~Compute the step size $t_k$ by solving the (quadratic) minimisation problem
\begin{equation*}
 \min_{t\in\mathbb{R}}\,\bigl\| x_k + t(Q(x_k)-x_k) - P_M(x_0) \bigr\|^2.
\end{equation*}

2.~Compute $x_{k+1}$ according to 
\begin{equation*}
  x_{k+1} := x_k+t_k\bigl(Q(x_k)-x_k\bigr). 
\end{equation*}
}
\end{algorithm}
When $M_1,\dots,M_n$ are linear subspaces, it can be shown (see Section~\ref{s:linear}) that the step size $t_k$ can be computed using the expression
  \begin{equation}
  \label{eq:tk intro}
   t_k = \frac{\langle x_k-Q(x_k),x_k\rangle}{\|x_k-Q(x_k)\|^2}\quad \text{if~}Q(x_k)\neq x_k. 
  \end{equation}
Since it only requires vector arithmetic, evaluating this expression comes with relatively low computational cost. Moreover, Gearhart \& Koshy's scheme gives the following refinement of the upper-bound provided by Theorem~\ref{th:affine rate} in \eqref{eq:linear bound}. 
\begin{theorem}[{Gearhart--Koshy~\cite{GK1989}}]\label{th:gk}
Let $M_1,\dots,M_n$ be closed affine subspaces of $\Hilbert$ with $M=\cap_{i=1}^nM_i\neq \emptyset$. For each sequence $(x_k)$ generated by Algorithm~\ref{a:gk}, there exists a sequence $(f_k)\subseteq[0,1]$ such that
\begin{equation*}
 \|x_k-P_M(x_0)\| 
 \leq c^k\left(\prod_{i=1}^kf_i\right)\|x_0-P_M(x_0)\|, 
\end{equation*}
where the constant $c\in[0,1]$ is given by \eqref{eq:c}.
\end{theorem}        
Although Theorem~\ref{th:gk} still holds for affine subspaces, the efficient expression for $t_k$ provided by \eqref{eq:tk intro} is only valid for linear subspaces (this will be explained more precisely Section~\ref{s:linear}). Thus, in the affine case, it is no longer obvious how to efficiently apply the scheme.

\medskip

In this work, we address the aforementioned problem by deriving an alternative expression for \eqref{eq:tk intro} which still holds in the affine case and still only requires vector arithmetic for its evaluation. Our key insight is the observation that \eqref{eq:tk intro} implicitly relies on the fact that the zero vector is always feasible for linear subspaces. The remainder of this work is structured as follows. In Section~\ref{s:prelim}, we collect the necessary preliminaries. In Section~\ref{s:linear}, we discuss Gearhart \& Koshy's derivation of \eqref{eq:tk intro} and, in Section~\ref{s:affine}, we provide an alternative formula which still holds in the affine setting. In Section~\ref{s:nonlinear}, we discuss some implications for nonlinear fixed iterations and finally, in Section~\ref{s:computational examples}, we provide computational examples.

\section{Preliminaries}\label{s:prelim}
Let $S\subseteq\Hilbert$ be a nonempty subset of $\Hilbert$. Recall that the \emph{(nearest point) projector} onto $S$ is the operator $P_S\colon\Hilbert\to S$ defined by
\begin{equation}\label{eq:PS defn}
  P_S(x) := \argmin_{z\in S}\|x-z\|.
\end{equation}
It is well-known (see, for instance, \cite[3.5]{D2012}) that $P_S$ is a well-defined operator whenever $S$ is closed and convex. Further, the definition in \eqref{eq:PS defn} also implies the \emph{translation formula}
\begin{equation}\label{eq:translation}
  P_{S}(x) = P_{S-y}(x-y) + y  \quad  \forall x,y\in\Hilbert,
\end{equation}
where $S-y:=\{s-y\in\Hilbert:s\in S\}$. The following proposition collects important properties of projectors for use in the subsequence sections.
\begin{proposition}[Properties of projectors]\label{prop:proj}
Let $S\subseteq\Hilbert$ be a nonempty, closed set.
\begin{enumerate}[(a)]
\item\label{prop:proj:convex} Suppose $S$ is convex. Then $p=P_S(x)$ if and only if
\begin{equation*}
  p\in S\text{~~and~~}\langle x-p,s-p\rangle\leq 0\quad\forall s\in S.
\end{equation*}
\item\label{prop:proj:affine} Suppose $S$ is an affine subspace. Then $p=P_S(x)$ if and only if
\begin{equation*}
  p\in S\text{~~and~~}\langle x-p,s-p\rangle= 0\quad\forall s\in S.
\end{equation*}
\item\label{prop:proj:linear} Suppose $S$ is a linear subspace. Then $P_S$ is a bounded, self-adjoint linear operator.
\end{enumerate}
\end{proposition}  
\begin{proof}
See, for instance, \cite[4.1]{D2012} for (a), \cite[9.26]{D2012} for (b), and \cite[5.13]{D2012} for (c).
\end{proof}

Let $S$ be a nonempty closed affine subspace and let $S'$ denote the associated linear subspace parallel to $S$. Then $S'$ can be expressed as $S'=S-y$ for any $y\in S$. In this case, the translation formula \eqref{eq:translation} implies
  \begin{equation}\label{eq:affine trans}
   P_{S}(x)=P_{S'}(x-s)+s\quad\forall x\in \Hilbert. 
  \end{equation}
By using Proposition~\ref{prop:proj}\eqref{prop:proj:linear}, this formula allows us to relate the affine projector $P_S$ to the self-adjoint operator $P_{S'}$. Furthermore, the characterisation in Proposition~\ref{prop:proj}\eqref{prop:proj:affine} is also equivalent to the condition $x-p\in (S')^\perp$ where the superscript ``$\perp$'' denotes the orthogonal complement of a subspace (see, for instance, \cite[9.26]{D2012}).

\section{Gearhart--Koshy Acceleration for Linear Subspaces}\label{s:linear}
In this section, we recall the derivation of Gearhart \& Koshy's scheme for linear subspaces~\cite{GK1989}. This serves to both introduce the scheme, and to highlighting the immediate difficulty with extending the result to affine spaces.

Denote $Q:=P_{M_n}\dots P_{M_1}$. Using this notation, the method of cyclic projection (as discussed in Theorem~\ref{th:mcp}) generates a sequence $(x_k)$ according to the fixed-point iteration
\begin{equation*}
 x_{k+1} := Q(x_k)\quad\forall k\in\mathbb{N}.
\end{equation*}
Gearhart \& Koshy's scheme attempts to accelerate convergence by instead defining the sequence  $(x_k)$ according to
\begin{equation*}
  x_{k+1} := x_k + t_k(Q(x_k)-x_k), 
\end{equation*}
where $t_k$ is chosen such that $x_{k+1}$ is the point in the affine span of $\{x_k,Q(x_k)\}$ closest to $P_M(x_0)$. In other words, $t_k$ is a solution to the (quadratic) minimisation problem
\begin{equation}\label{eq:tk def}
  \min_{t\in\mathbb{R}}\bigl\| x_k + t(Q(x_k)-x_k) - P_M(x_0) \bigr\|^2.
\end{equation}
Using the first-order optimality condition, we deduce that a solution to \eqref{eq:tk def} is given by
\begin{equation}\label{eq:tk}
  t_k = \begin{cases}
    \frac{\langle x_k-Q(x_k), x_k-P_M(x_0)\rangle}{\|x_k-Q(x_k)\|^2} & \text{if~}Q(x_k)\neq x_k \\
    1 & \text{otherwise.}
    \end{cases}
\end{equation}
It is worth noting that so-far the derivation of \eqref{eq:tk} has not relied on any properties of the sets other than the fact that $Q$ is well-defined. However, as written,  \eqref{eq:tk} does not provide a useful expression for computing $t_k$ since, when $Q(x_k)\neq x_k$, it requires knowledge of $P_M(x_0)$ (\emph{i.e.,} the solution to the problem we are trying to solve).

 In the case when $M_1,\dots,M_n$ are linear subspaces, this difficulty can be overcome by using self-adjointess of the projectors (Proposition~\ref{prop:proj}\eqref{prop:proj:linear}). Indeed, since $Q^*P_M = P_{M_1}\dots P_{M_n}P_M=P_M$, we have
\begin{multline}\label{eq:gk tk}
 \langle x_k-Q(x_k), P_M(x_0)\rangle \\
 = \langle x_k, P_M(x_0)\rangle-\langle x_k, Q^*P_M(x_0)\rangle =0. 
\end{multline}
Combining \eqref{eq:tk} with \eqref{eq:gk tk} then gives
\begin{equation}\label{eq:tk subspace}
  t_k = \begin{cases}
    \frac{\langle x_k-Q(x_k), x_k\rangle}{\|x_k-Q(x_k)\|^2} & \text{if~}Q(x_k)\neq x_k \\
    1 & \text{otherwise.}
    \end{cases}
\end{equation}
This expression no longer requires knowledge of $P_M(x_0)$, and can be evaluated using vector arithmetic and the current iterate. Explicitly, we have the following fully-explicit version of Algorithm~\ref{a:gk} for linear subspaces.

\begin{algorithm}[ht]
\caption{Gearhart--Koshy acceleration with linear subspaces.}
\textbf{Initialisation.}~An initial point $x_0\in\Hilbert$.\\
\For{$k=0,1,2,\dots$}{
1.~Compute the step size $t_k$ using \eqref{eq:tk subspace}.

2.~Compute $x_{k+1}$ according to 
\begin{equation*}
  x_{k+1} := x_k+t_k\bigl(Q(x_k)-x_k\bigr).
\end{equation*}
}
\end{algorithm}

We now highlight the difficulty in using \eqref{eq:tk subspace} for affine subspaces. To this end, assume that $M_1,\dots,M_n$ are affine subspaces with parallel linear subspaces denoted $M_1',\dots,M_n'$. Further let $M'$ denote the linear subspaces parallel to the affine subspace $M$. As a consequence of the translation formula \eqref{eq:affine trans}, for $m_i\in M_i$, $m\in M$ and $x\in\Hilbert$, we have
\begin{equation}\label{eq:translation P_M}
\begin{aligned}
 P_{M_i}(x) &= P_{M_i'}(x-m_i)+m_i\text{~~and~~}\\
 P_{M}(x) &= P_{M'}(x-m)+m.
\end{aligned}
\end{equation}
We now attempt an argument analogous to \eqref{eq:gk tk} by reduction to the linear case using the translation formulae \eqref{eq:translation P_M}. To this end, let $m\in M=\cap_{i=1}^nM_i$ and denote
\begin{equation*}
  Q':=P_{M_n'}\dots P_{M_1'}.
\end{equation*}
Then, for all $x\in\Hilbert$, \eqref{eq:translation P_M} implies
\begin{equation}\label{eq:Q'}
\begin{aligned}
  Q(x) &= \bigl(P_{M_n}\dots P_{M_3}P_{M_2}\bigr)\bigl( P_{M_1'}(x-m)+m\bigr) \\
       &= \bigl(P_{M_n}\dots P_{M_3}\bigr)\bigl( P_{M_2'}P_{M_1'}(x-m)+m\bigr) \\
       &\;\;\vdots \\
       &= \bigl(P_{M_n'}\dots P_{M_3'}P_{M_2'}P_{M_1'}\bigr)(x-m)+m \\
       &= Q'(x-m)+m.
\end{aligned}
\end{equation}
Noting that $(Q')^*P_{M'}=P_{M'}$, we may express the term involving $P_M(x_0)$ in \eqref{eq:tk} as
\begin{equation}\label{eq:tk derivation}
\begin{aligned}
 &\langle x_k-Q(x_k),P_M(x_0)\rangle \\
  & = \langle (x_k-m)-Q'(x_k-m),P_{M'}(x_0-m)+m\rangle \\
  &=\langle x_k-m,P_{M'}(x_0-m)-(Q')^*P_{M'}(x_0-m)\rangle
  \\
  & \qquad+\langle (x_k-m)-Q'(x_k-m),m\rangle \\
  &= \langle x_k-Q(x_k),m\rangle.
\end{aligned}
\end{equation}
When $Q(x_k)\neq x_k$, substituting this expression into \eqref{eq:tk} gives
  \begin{equation}\label{eq:tk naive}
   t_k = \frac{\langle x_k-Q(x_k), x_k-m\rangle}{\|x_k-Q(x_k)\|^2}.
  \end{equation}
Thus $t_k$ could be computed using \eqref{eq:tk naive} whenever an intersection point $m\in M$ is known. In particular, when $M$ is a linear subspace, taking $m=0\in M$ recovers the original Gearhart--Koshy formula \eqref{eq:tk subspace}. In this sense, the derivation of  \eqref{eq:tk subspace} implicitly uses the fact that linear subspaces always contain the zero vector. For the general problem however, finding an intersection point $m\in M$ can be as hard as solving the best approximation problem \eqref{eq:bap} itself. Thus in practice, \eqref{eq:tk naive} is generally not of much use.

\section{Gearhart--Koshy Acceleration for Affine Subspaces}\label{s:affine}
In this section, we derive an alternate expression for the step size $t_k$ in Gearhart \& Koshy's scheme which is still valid for affine subspaces and which can be explicitly computed without knowledge of an intersection point (unlike the expression in \eqref{eq:tk naive}). To this end, let $Q_i\colon\Hilbert\to\Hilbert$ denote the operator
\begin{equation*}
 Q_i := \begin{cases}
                P_{M_i}\dots P_{M_1} & \text{if~}i\in\{1,\dots,n\} \\
                I & \text{if~}i=0. \\                
             \end{cases}
\end{equation*}
                        
\begin{lemma}\label{l:affine tk}
Let $M_1,\dots,M_n$ be closed affine subspaces of $\Hilbert$ with $M=\cap_{i=1}^nM_i\neq \emptyset$. For $i\in\{1,\dots,n\}$, let $M_i'$ denote the linear subspace parallel to $M_i$. If $Q(x_k)\neq x_k$, then the solution of \eqref{eq:tk def} is given by
\begin{equation}\label{eq:good tk}
 t_k = \frac{1}{2} + \frac{\sum_{i=1}^n\|Q_{i-1}x_k-Q_{i}x_k\|^2}{2\|x_k-Q(x_k)\|^2}. 
\end{equation}
\end{lemma}           
\begin{proof}
Let $m\in M=\cap_{i=1}^nM_i$. By the argument in \eqref{eq:tk derivation}, we have
\begin{equation}\label{eq:l1}
   \langle x_k-Q(x_k),P_M(x_0)\rangle = \langle x_k-Q(x_k),m\rangle.
\end{equation}
Since $Q_i(x_k)-m\in M'_i$ and $\range(I-P_{M_i})\subseteq (M_i')^\perp$, for all $i\in\{1,\dots,n\}$, we have
\begin{equation}\label{eq:l2}
    \langle Q_i(x_k)-m,(I-P_{M_i})Q_{i-1}x_k \rangle = 0.
\end{equation}
By combining \eqref{eq:l1} and \eqref{eq:l2}, we therefore obtain
\begin{align*}
 &\langle x_k-Q(x_k),P_M(x_0)\rangle \\
  &= \langle m, x_k-Q(x_k) \rangle \\ 
  &= \sum_{i=1}^n\langle m,(I-P_{M_i})Q_{i-1}(x_k) \rangle \\  
  &= \sum_{i=1}^n\langle Q_i(x_k),(I-P_{M_i})Q_{i-1}(x_k) \rangle \\
  &= \frac{1}{2}\sum_{i=1}^n\left( \|Q_{i-1}(x_k)\|^2-\|Q_i(x_k)\|^2 \right.\\
  &\qquad\qquad \left.-\|Q_{i-1}(x_k)-Q_i(x_k)\|^2\right) \\
&= \frac{1}{2}\|x_k\|^2-\frac{1}{2}\|Q(x_k)\|^2 \\
&\qquad\qquad - \frac{1}{2}\sum_{i=1}^n\|Q_{i-1}(x_k)-Q_i(x_k)\|^2.
\end{align*}
Together with \eqref{eq:tk}, this yields
  \begin{align*}
  &2t_k\|x_k-Q(x_k)\|^2 \\
  &= 2\langle x_k-Q(x_k),x_k\rangle - 2\langle x_k-Q(x_k),P_M(x_0)\rangle \\
  &= \left(\|x_k\|^2 + \|x_k-Q(x_k)\|^2 - \|Qx_k\|^2\right) \\
  &\qquad\quad - 2\langle x_k-Q(x_k),P_M(x_0)\rangle \\
  &= \|x_k-Q(x_k)\|^2 + \sum_{i=1}^n\|Q_{i-1}x_k-Q_{i}x_k\|^2,
  \end{align*}
from which the claimed result follows.  
\end{proof}

\begin{algorithm}[ht]
\caption{Gearhart--Koshy acceleration for affine subspaces.}\label{a:gk affine}
\textbf{Initialisation.}~An initial point $x_0\in\Hilbert$.\\
\For{$k=0,1,2,\dots$}{
1.~Compute the step size $t_k$ using
\begin{equation*}
\hspace{-0.4cm}\small t_k = \begin{cases}
  \frac{1}{2} + \frac{\sum_{i=1}^n\|Q_{i-1}x_k-Q_{i}x_k\|^2}{2\|x_k-Q(x_k)\|^2} & \text{if~}Q(x_k)\neq x_k \\
   1 & \text{otherwise.}
           \end{cases}
\end{equation*}

2.~Compute $x_{k+1}$ according to 
\begin{equation}\label{eq:xk}
x_{k+1} := x_k+t_k\bigl(Q(x_k)-x_k\bigr).
\end{equation}
}
\end{algorithm}
\begin{theorem}[Gearhart--Koshy acceleration for affine subspaces]\label{th:affine gk}
Let $M_1,\dots,M_n$ be closed affine subspaces of $\Hilbert$ with $M=\cap_{i=1}^nM_i\neq \emptyset$. For each sequence $(x_k)$ generated by Algorithm~\ref{a:gk affine}, there exists a sequence $(f_k)\in[0,1]$ such that 
  \begin{equation}\label{eq:affine bound}
   \| x_{k}-P_M(x_0) \| \leq \|x_0-P_M(x_0)\| \left(\prod_{i=1}^kf_i\right)c^k \quad\forall k\in\mathbb{N},
  \end{equation}
where the constant $c\in[0,1]$ is given by \eqref{eq:c}.  
\end{theorem}
\begin{proof}
Let $m\in M=\cap_{i=1}^nM_i$ and denote $x_k':=x_k-m$ for all $k\in\mathbb{N}$. Then \eqref{eq:xk} together with \eqref{eq:Q'} implies
 \begin{equation}\label{eq:xk'}
   x_{k+1}'=x_k' + t_k\bigl(Q'(x_k')-x_k'\bigr)\quad\forall k\in\mathbb{N}. 
 \end{equation}
By Lemma~\ref{l:affine tk}, $t_k$ given by \eqref{eq:good tk} is the solution to \eqref{eq:tk def}. Hence, the iteration \eqref{eq:xk'} coincides with Gearhart \& Koshy's scheme applied to the linear subspaces $M_1',\dots,\allowbreak M_n'$. The claimed result thus follows from Theorem~\ref{th:gk}, noting that $x_k=x_k'+m$ for all $k\in\mathbb{N}$.
\end{proof}

We note that although Gearhart \& Koshy's scheme is an attempt to accelerate convergence, Theorems~\ref{th:gk} and \ref{th:affine gk} do not necessarily imply that the sequence $(x_k)$ converges faster. Rather, the theory implies that the scheme improves the upper bound on the rate of convergence provided by \eqref{eq:linear bound} and \eqref{eq:affine bound}, respectively. Nevertheless, when $n=2$, the scheme does indeed accelerate convergence, see \cite[Theorem~3.23]{BDHP2003}. On the other hand, when $n\geq 3$, the scheme can actually be slower, see \cite[Example~3.24]{BDHP2003}. 

To overcome this, Bauschke, Deutsch, Hundal and Park studied a symmetrised version of the method of cyclic projections based on the operator $S\colon\Hilbert\to\Hilbert$ given by
\begin{equation*}
  S:= P_{M_1}P_{M_2}\dots P_{M_{n-1}}P_{M_n}P_{M_{n-1}}\dots P_{M_1}. 
\end{equation*}
The \emph{method of symmetric cyclic projections} is the corresponding fixed point iteration given by $x_{k+1} := S(x_k)$ for all $k\in\mathbb{N}$. When the sets are linear subspaces, the operator $S=Q^*Q$ has better properties than $Q$. For instance, $S$ is self-adjoint and nonnegative (\emph{i.e.,} $\langle Sx,x\rangle\geq 0$ for all $x\in\Hilbert$) whereas the operator $Q$ is usually not. On the other hand, its evaluation requires computing $n-1$ additional projections.

For $i\in\{1,\dots,2n\}$, define the operator $S_i\colon\Hilbert\to\Hilbert$ by
\begin{equation*}
 S_i := \begin{cases}
           P_{M_i}P_{M_{i-1}}\dots P_{M_1} & 1\leq i\leq n, \\
           P_{M_{(2n-i)}}\dots P_{M_{n-1}} P_{M_n}S_n & n+1\leq i\leq 2n-1, \\
           I & i=0.            
          \end{cases} 
\end{equation*}
                   
The following theorem, which extends \cite[Corollary~3.21]{BDHP2003} to the affine case, shows that the Gearhart--Koshy-type acceleration of method of symmetric cyclic projections is at least as fast as method of symmetric cyclic projections. The resulting algorithm is summarised in Algorithm~\ref{a:gk symmetric cyclic}.

\begin{theorem}[Accelerated symmetric cyclic projections]\label{th:accel symcp}
Let $M_1,\dots,M_n$ be closed affine subspaces of $\Hilbert$ with $M=\cap_{i=1}^nM_i\neq \emptyset$. Then the sequence $(z_k)$ generated by Algorithm~\ref{a:gk symmetric cyclic} satisfies
  \begin{equation*}
   \| z_{k}-P_M(x_0) \| \leq \|S^{k+1}(x_0)-P_M(x_0)\| \quad\forall k\in\mathbb{N}.
  \end{equation*}
Thus, the accelerated sequence $(z_k)$ converges at least as fast as the unaccelerated symmetric cyclic projection sequence.
\end{theorem}
\begin{proof}
First note that the symmetrised operator $S$ coincides with its non-symmetric counterpart $Q$ applied to the $2n-1$ sets $M_1\dots,M_{n-1},M_n,M_{n-1},\dots M_1$. Consequently, Lemma~\ref{l:affine tk} implies that $s_k$ in \eqref{eq:sk} is a solution to the problem
\begin{equation*}
  \min_{s\in\mathbb{R}}\bigl\| z_k+s_k\bigl(S(z_k)-z_k\bigr)-P_M(z_0) \|^2. 
\end{equation*}
The result then follows by a translation argument together with \cite[Corollary~3.21]{BDHP2003}.
\end{proof}

\begin{algorithm}[ht]
\caption{Accelerated symmetric cyclic projections for affine subspaces.}\label{a:gk symmetric cyclic}
\textbf{Initialisation.}~Given an initial point $x_0\in\Hilbert$, set $z_0:=S(x_0)$.\\
\For{$k=0,1,2,\dots$}{
1.~Compute the step size $t_k$ using 
  \begin{equation}\label{eq:sk}
  s_k = \frac{1}{2} + \frac{\sum_{i=1}^{2n}\|S_{i-1}(z_k)-S_i(z_k)\|^2}{2\|z_k-S(z_k)\|^2}
  \end{equation}
  when $S(z_k)\neq z_k$, else set $s_k=1$.
2.~Compute $z_{k+1}$ according to 
\begin{equation*}
  z_{k+1} := z_k+s_k\bigl(S(z_k)-z_k\bigr). 
\end{equation*}
}
\end{algorithm}

\section{Extensions to Firmly Nonexpansive Operators}\label{s:nonlinear}
The orthogonality condition \eqref{eq:l2} was a key ingredient in the proof of Lemma~\ref{l:affine tk}. In this section, we investigate what remains true without this property. Our focus will be the following class of operators which generalise affine projectors.
\begin{definition}\label{def:fqne}
An operator $T\colon\Hilbert\to\Hilbert$ is \emph{firmly quasi-nonexpansive} if
 \begin{equation}\label{eq:fqne1}
  \|T(x)-y\|^2 + \|x-T(x)\|^2 
  \leq \|x-y\|^2
 \end{equation}
for all $x\in\Hilbert$ and $y\in \Fix T:=\{y\in\Hilbert:T(y)=y\}$.
\end{definition}
It is straightforward to check that the inequality \eqref{eq:fqne1} is equivalent to requiring
\begin{equation}\label{eq:fqne}
  0 \leq \langle T(x)-y,x-T(x)\rangle \quad \forall x\in\Hilbert,\,\forall y\in\Fix T.
\end{equation}
As a consequence of Proposition~\ref{prop:proj}\eqref{prop:proj:convex}, projectors onto convex sets are firmly quasi-nonexpansive. And, in particular, Proposition~\ref{prop:proj}\eqref{prop:proj:affine} shows that \eqref{eq:fqne} holds with equality when $T$ is a projector onto an affine set. More generally, it can be seen that \eqref{eq:fqne} (as well as \eqref{eq:fqne1}) holds with equality when $2T-I$ preserves distances to fixed points in the sense that
\begin{equation*}
 \|(2T-I)(x)-y\|=\|x-y\|\quad \forall x\in\Hilbert,\,\forall y\in\Fix T.
\end{equation*}
Another example of a firmly quasi-nonexpansive operator satisfying this problem is the \emph{Douglas--Rachford operator} $T_{C_1,C_2}\colon\Hilbert\to\Hilbert$ defined by 
\begin{equation}\label{eq:dr op}
 T_{C_1,C_2} := \frac{1}{2}\left(I+R_{C_2}R_{C_1}\right), 
\end{equation}
when the sets $C_1,C_2\subseteq\Hilbert$ are closed affine subspaces. Here $R_{C_i}:=2P_{C_i}-I$ denotes the \emph{reflector} with respect to $C_i$. This fact can be verified by noting that $2T_{C_1,C_2}-I=(2P_{C_2}-I)(2P_{C_1}-I)$ and applying Proposition~\ref{prop:proj}\eqref{prop:proj:affine}.

Let $T_1,\dots,T_n\colon\Hilbert\to\Hilbert$ be firmly quasi-nonexpansive operators with $\cap_{i=1}^n\Fix T_i\neq\emptyset$. Denote
  $ Q := T_n\dots T_2T_1. $
Given an initial point $x_0\in\Hilbert$, the iteration
  \begin{equation}\label{eq:fqne iterations}
   x_{k+1}:=Q(x_k)\quad\forall k\in\mathbb{N},
  \end{equation}
can be shown to converge weakly to a solution of the \emph{common fixed point problem}
\begin{equation}\label{eq:cfpp}
  \text{find~}x\in\bigcap_{i=1}^n\Fix T_i=\Fix Q,
\end{equation}
where we note that the equality in \eqref{eq:cfpp} follows from \cite[Corollary~4.50]{BC2ed}.

In an attempt to accelerate \eqref{eq:fqne iterations}, we consider schemes of the form
\begin{equation*}
   x_{k+1} = x_k + t_k\left( Q(x_k)-x_k \right),
\end{equation*}
where $m_k\in\cap_{i=1}^n\Fix T_i$ and $t_k$ is the solution to the problem
\begin{equation*}
  \min_{t\in\mathbb{R}}\left\|x_k + t\left( Q(x_k)-x_k \right)-m_k\right\|^2. 
\end{equation*}
In other words, when $x_k\not\in\cap_{i=1}^n\Fix T_i$,  $t_k$ is given by
\begin{equation*}
  t_k = \frac{\langle x_k-Q(x_k),x_k-m_k\rangle }{\|x_k-Q(x_k)\|^2}.
\end{equation*}
The following proposition, which can be viewed as a firmly nonexpansive analogue of Lemma~\ref{l:affine tk}, provides a lower bound for the value of $t_k$. It is worth noting that this lower bound is independent of the choice of intersection point $m_k\in\cap_{i=1}^n\Fix T_i$. 
\begin{proposition}[Acceleration step size lower bound]\label{prop:bound}
Let $T_1,\dots,T_n\colon\Hilbert\to\Hilbert$ be firmly quasi-nonexpansive operators and let $m\in\cap_{i=1}^n\Fix T_i\neq\emptyset$. If $Q(x_k)\neq x_k$, then the solution to the minimisation problem 
  \begin{equation}\label{eq:t min}
   \min_{t\in\mathbb{R}}\left\|x_k + t\left( Q(x_k)-x_k \right)-m\right\|^2
  \end{equation}
satisfies 
\begin{equation}\label{eq:t}
 t \geq \frac{1}{2} + \frac{\sum_{i=1}^n\|Q_{i-1}(x_{k-1})-Q_{i}(x_{k-1})\|^2}{2\|x_{k-1}-Qx_{k-1}\|^2}. 
\end{equation}
Furthermore, if $T_1,\dots, T_n$ satisfy \eqref{eq:fqne1} with equality, then \eqref{eq:t} also holds with equality.
\end{proposition}  
\begin{proof}
Since $m\in \cap_{i=1}^n\Fix T_i$, \eqref{eq:fqne} implies
\begin{equation}\label{eq:fqne bound}
\begin{aligned}
&\langle x_k-Q(x_k),m\rangle\\
 &= \sum_{i=1}^n\langle m,(I-T_i)Q_{i-1}(x_k)\rangle \\
 &\leq \sum_{i=1}^n\langle T_iQ_{i-1}(x_k),(I-T_i)Q_{i-1}(x_k)\rangle \\
 &= \frac{1}{2}\sum_{i=1}^n\left( \|Q_{i-1}(x_k)\|^2 - \|Q_i(x_k)\|^2 \right.\\
 &\qquad\qquad\qquad\left. - \|Q_{i-1}(x_k)-Q_i(x_k)\|^2 \right) \\
 &= \frac{1}{2}\|x_k\|^2-\frac{1}{2}\|Q(x_k)\|^2\\
  &\qquad\qquad\qquad -\frac{1}{2}\sum_{i=1}^n\|Q_{i-1}(x_k)-Q_i(x_k)\|^2.
\end{aligned}
\end{equation}
Using the optimality conditions for \eqref{eq:t min}, followed by applying \eqref{eq:fqne bound} yields
\begin{align*}
&2t\|x_k-Q(x_k)\|^2 \\
&= 2\langle x_k-Q(x_k),x_k\rangle - 2\langle x_k-Q(x_k),m\rangle \\
&= \left( \|x_k\|^2+\|x_k-Q(x_k)\|^2-\|Q(x_k)\|^2 \right)\\
&\qquad  - 2\langle x_k-Q(x_k),m\rangle \\
&\leq \|x_k-Q(x_k)\|^2 + \sum_{i=1}^n\|Q_{i-1}(x_k)-Q_i(x_k)\|^2.
\end{align*}
The claimed result then follows by rearranging this expression. Furthermore, when $T_1,\dots, T_n$ satisfy \eqref{eq:fqne1} with equality, \eqref{eq:fqne bound} holds with equality and hence so does \eqref{eq:t}.
\end{proof}

This observation allows us to apply the acceleration technique to affine settings beyond projectors including the Douglas--Rachford variants studied in \cite{BT2014,BT2015,T2016,CM2016}. The simplest realisation is the symmetrised Douglas--Rachford algorithm considered below. 

\begin{proposition}\label{prop:DR op}
Let $M_1,M_2\subseteq\Hilbert$ be closed affine subspaces with $M_1\cap M_2\neq\emptyset$ and parallel linear subspaces denoted $M_1'$ and $M_2'$, respectively. Consider the operators $T,T'\colon\Hilbert\to\Hilbert$ given by
\begin{equation*}
 T:=T_{M_2,M_1}T_{M_1,M_2}\quad\text{and}\quad T':=T_{M_2',M_1'}T_{M_1',M_2'}.
\end{equation*}
Then the following assertions hold. 
\begin{enumerate}[(a)]
\item\label{prop:DR op a} $T(x) = T'(x-m)+m$ for all $x\in\Hilbert$ and $m\in M_1\cap M_2$.
\item\label{prop:DR op b} $(T_{M_1',M_2'})^\ast=T_{M_2',M_1'}$ and $(T_{M_2',M_1'})^\ast=T_{M_1',M_2'}$.
\item\label{prop:DR op c} $T'$ is self-adjoint and nonnegative (i.e., $\langle x,T'(x)\rangle\geq 0$ for all $x\in\Hilbert$).
\item\label{prop:DR op d} $\Fix T=\Fix T_{M_1,M_2}\cap\Fix T_{M_2,M_1}=M_1\cap M_2+(M_1-M_2)^\perp=M_1\cap M_2+(M_1')^\perp\cap(M_2')^\perp$.
\item\label{prop:DR op e} $P_{M_1}P_{\Fix T}=P_{M_2}P_{\Fix T}=P_{M_1\cap M_2}$.
\end{enumerate}
\end{proposition}
\begin{proof}
\eqref{prop:DR op a}:~See the proof of \cite[Theorem~4.1]{BT2014}.
\eqref{prop:DR op b}:~By linearity of adjoints and self-adjointness of the projectors onto linear subspaces (Proposition~\ref{prop:proj}\eqref{prop:proj:linear}), we have $R_{M_1}^\ast=R_{M_1}$, $R_{M_2}^\ast=R_{M_2}$ and
\begin{equation*}
   (T_{M_1',M_2'})^\ast = \frac{I+R_{M_1'}^*R_{M_2'}^*}{2}=\frac{I+R_{M_1'}R_{M_2'}}{2}=T_{M_2',M_1'}.
\end{equation*}  
\eqref{prop:DR op c}:~Using \eqref{prop:DR op b}, we deduce that
\begin{equation*}
 (T')^\ast=(T_{M_1',M_2'})^\ast(T_{M_2',M_1'})^\ast  = T_{M_2',M_1'}T_{M_1',M_2'}=T'
\end{equation*}
and 
\begin{multline*} \langle x,T'(x)\rangle = \langle (T_{M_2',M_1'})^*(x),T_{M_1',M_2'}(x)\rangle \\
=\|T_{M_1',M_2'}x\|^2\geq 0\quad\forall x\in\Hilbert.
\end{multline*}
\eqref{prop:DR op d}:~For the first equality, see the proof of \cite[Theorem~2.4.5]{T2016}. Next, let $m\in M_1\cap M_2$ and note that 
$ M_1-M_2=(M_1-m)-(M_2-m)=M_1'-M'_2=M_1'+M_2'$. In particular, $M_1-M_2$ is a linear subspace and $M_1-M_2=M_1'+M_2'=M_2-M_1$. Furthermore,  $(M_1'+M_2')^\perp =(M_1')^\perp\cap(M_2')^\perp$ (see \cite[Proposition~6.27]{BC2ed}). Appealing to \cite[Corollary~3.9]{BCL2004} yields
\begin{align*}
 \Fix T_{M_1,M_2}
 &= M_1\cap M_2+(M_1-M_2)^\perp, \\
 \Fix T_{M_2,M_1}&= M_2\cap M_1+(M_2-M_1)^\perp.
\end{align*}
The second and third equalities now follow.
\eqref{prop:DR op e}:
 Let $x\in\Hilbert$ and denote $p:=P_{\Fix T}(x)$. By \cite[Lemma~2.4.4]{T2016}, we have $P_{M_1}(p)=P_{M_2}(p)\in M_1\cap M_2$. 
Let $m\in M_1\cap M_2\subseteq\Fix T$ be arbitrary. By Proposition~\ref{prop:proj}\eqref{prop:proj:affine} applied to $P_{\Fix T}$, $P_{\Fix T}$ and $P_{M_1}$, respectively, we have
\begin{multline*}
  \langle x-P_{M_1}(p),m-P_{M_1}(p)\rangle \\
  =  \langle x-p,m-p\rangle + \langle x-p,p-P_{M_1}(p)\rangle \\
  +   \langle p-P_{M_1}(p),m-P_{M_1}(p)\rangle=0.
\end{multline*}
This shows that $P_{M_1}(p)=P_{M_1\cap M_2}(x)$ and hence completes the proof.
\end{proof}

\begin{theorem}[Accelerated symmetric Douglas--Rachford]
Let $M_1,M_2\subseteq\Hilbert$ be closed affine subspaces with $M:=M_1\cap M_2\neq\emptyset$. Let $T\colon\Hilbert\to\Hilbert$ denote the symmetric Douglas--Rachford operator given by 
\begin{equation*}
   T:=T_{M_2,M_1}T_{M_1,M_2}.
\end{equation*}
Then the sequence $(z_k)$ generated by Algorithm~\ref{a:dr} satisfies
\begin{equation}\label{eq:sDR rate}
\|z_k-P_{\Fix T}(x_0)\| \leq \|T^{k+1}(x_0)-P_{\Fix T}(x_0)\| \quad \forall k\in\mathbb{N}. \end{equation}
Thus, the accelerated sequence $(z_k)$ converges at least as fast as the unaccelerated symmetric Douglas--Rachford sequence. Moreover, we have
\begin{multline}\label{eq:sDR soln}
 \hspace{-0.5cm}\max\left\{ \|P_{M_1}(z_k)-P_M(x_0)\|,\|P_{M_2}(z_k)-P_M(x_0)\| \right\} \\
 \leq \|z_k-P_{\Fix T}(x_0)\|. 
\end{multline}
\end{theorem}
\begin{proof}
According to the discussion after \eqref{eq:dr op}, the operators $T_{M_2,M_1}$ and $T_{M_1,M_2}$ both satisfy \eqref{eq:fqne1} in Definition~\ref{def:fqne} with equality and thus Proposition~\ref{prop:bound} implies that $t_k$ given by \eqref{eq:sDR tk} satisfies 
\begin{equation*}
   t_k = \argmin_{t\in\mathbb{R}}\|z_k+t\bigl(T(z_k)-z_k\bigr)-P_{\Fix T}(x_0)\|^2.
\end{equation*}   
Let $m\in M\subseteq\Fix T$, denote $T':= T_{M_2',M_1'}T_{M_1',M_2'}$ and denote $z_k'=z_k-m$ for all $k\in\mathbb{N}$. By Proposition~\ref{prop:DR op}\eqref{prop:DR op c}, $T'$ is self-adjoint and nonnegative. By Proposition~\ref{prop:DR op}\eqref{prop:DR op a}, we have
\begin{equation*}
  z_{k+1}' = z_k' +t_k\big(T'(z_k')-z_k'\bigr)\quad\forall k\in\mathbb{N}.
\end{equation*}
In other words, the sequence $(z_k')$ coincides with the sequence in \cite[Theorem~3.20]{BDHP2003} applied to the linear subspaces $M_1'$ and $M_2'$.
Consequently, applying \cite[Theorem~3.20]{BDHP2003}, followed by a translation argument, yields \eqref{eq:sDR rate}. Inequality~\eqref{eq:sDR soln} then follows from firm quasi-nonexpansivity of $P_{M_1}$ and $P_{M_2}$ combined with Proposition~\ref{prop:DR op}\eqref{prop:DR op e}.
\end{proof}

\begin{algorithm}[ht]
\caption{Accelerated symmetric Douglas--Rachford for affine subspaces.}\label{a:dr}
\textbf{Initialisation.}~Given an initial point $x_0\in\Hilbert$, set $z_0:=T(x_0)$.\\
\For{$k=0,1,2,\dots$}{
1.~Compute the step size $t_k$ using
\begin{multline}\label{eq:sDR tk}
 \text{\hspace{-1cm}\small $t_k = \textstyle\frac{1}{2} + \frac{\|z_k-T_{M_1,M_2}(z_k)\|^2+\|T_{M_1,M_2}(z_k)-T(z_k)\|^2}{2\|z_k-T(z_k)\|^2}$}
\end{multline}
when $T(z_k)\neq z_k$, else set $t_k=1$.

2.~Compute $z_{k+1}$ according to 
\begin{equation*}
z_{k+1} = z_k + t_k(T(z_k)-z_k).
\end{equation*}
}
\end{algorithm}  

\section{Computational Examples}\label{s:computational examples}
In this section, we provide numerical examples to demonstrate the results from the previous sections. Our presentation will focus on the comparison between the method of cycling projections (Theorem~\ref{th:affine rate}) and its accelerated counterpart (Theorem~\ref{th:affine gk}). However analogous conclusion apply for the other methods considered in this paper. All computations were performed in Python 3 on a machine running Ubuntu 18.04 with an Intel Core i7-8665U and 16GB of memory.

Since the bound in Theorems~\ref{th:affine rate}~\&~\ref{th:affine gk} depends on the Friederichs angles between the constraints, we began by studying the effect of this angle on the effectiveness of the acceleration in a simple setting. To this end, let $\Hilbert=\mathbb{R}^2$, $x^*\in\Hilbert$, and $\theta\in(0,\pi/2)$. Consider the affine subspaces given by
\begin{equation}\label{eq:2Dexample}
\begin{aligned}
 M_1 &:=\{(x_1,0)+x^*:x_1\in\mathbb{R}\},\\
 M_2 &:=\{(r\cos \theta,r\sin\theta)+x^*:r\in\mathbb{R}\}.
\end{aligned}
\end{equation}
Then $M_1\cap M_2=\{x^*\}$ and the cosine of Friederichs angle between the corresponding parallel subspaces is equal to $\cos\theta$. Furthermore, $P_{M_1\cap M_2}(x)=x^*$ for all $x\in\Hilbert$.

\begin{figure}
  \centering
  \includegraphics[scale=0.5]{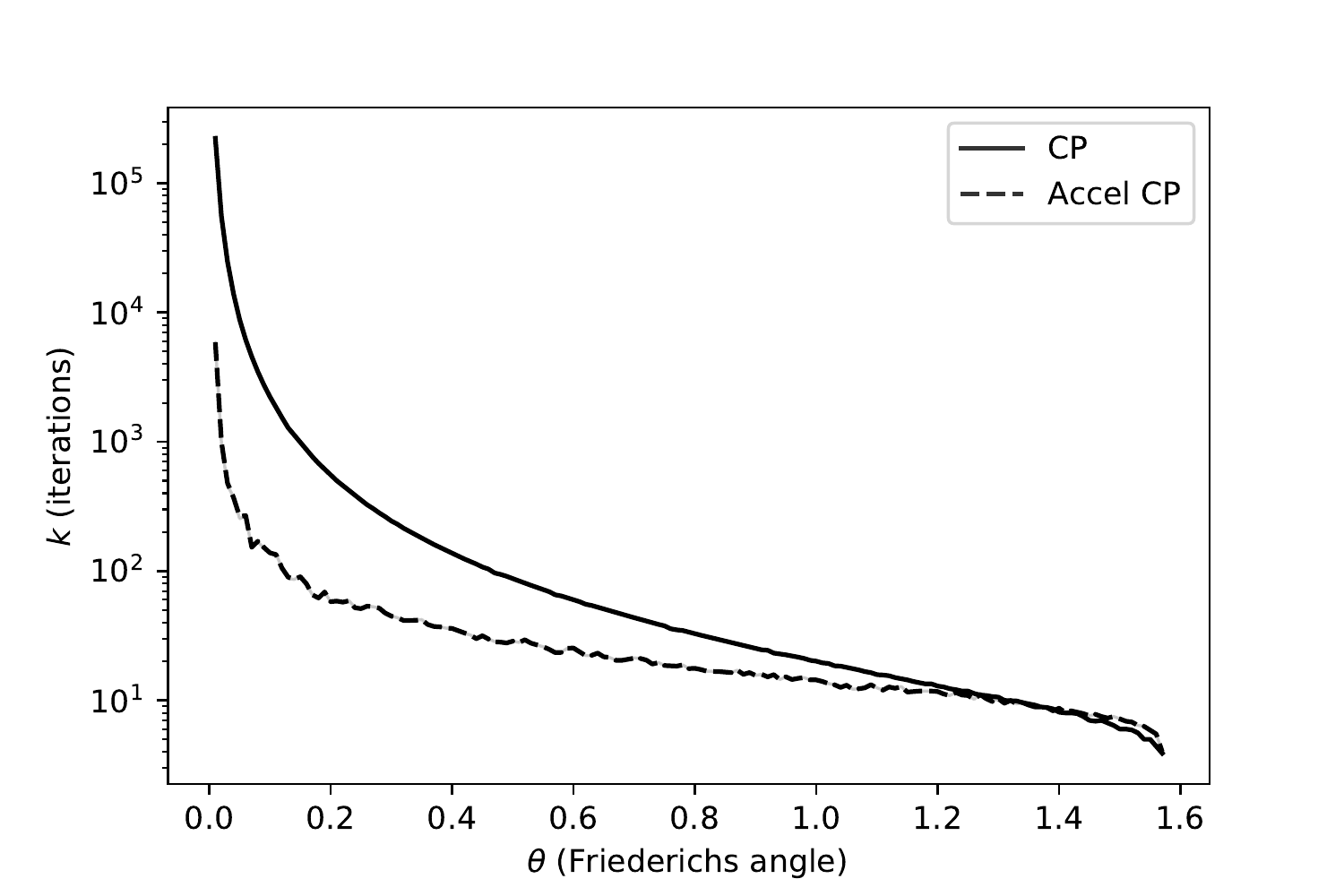}
  \caption{\rmfamily The effect of the Friederichs angle, $\theta$, on iterations, $k$, needed to reach $\|x_k-P _{M_1\cap M_2}(x_0)\|<10^{-9}$ for the method of cyclic projections and its acceleration.\label{fig:friederichs_a}}
\end{figure}

We compared the method of cyclic projection and its accelerated counterpart by applying them to the best approximation problem specified by the constraints in \eqref{eq:2Dexample}. Instances of this problem were generated by randomly choosing $x^*\in\Hilbert$ and choosing $\theta\in\{0.01,0.02,\allowbreak\dots,1.57\}$ (note that $1.57\approx\frac{\pi}{2}$). For each instance, ten replications of each algorithm were performed, starting from different random points $x_0\in\Hilbert$ having $\|x_0\|=10$. The average number of iterations required to trigger the termination criteria 
\begin{equation*}
\|x_k-P_{M_1\cap M_2}(x_0)\|<\epsilon\text{~~with~~}\epsilon=10^{-9}
\end{equation*}
as a function of the Friederichs angle are reported in Figure~\ref{fig:friederichs_a}.

This figure suggests that the acceleration is very effective for small Friederichs angles, where it is approximately $100$ times better in terms of iterations. Its effectiveness decreases with increasing Friederichs angle, with no significant improvement being provided for $\theta\approx\pi/2$.

Our next example concerns solving the best approximation problem subject to 
a linear system of the form $Ax=b$ where $A\in\mathbb{R}^{n\times m}$, $b\in\mathbb{R}^m$ and $n<m$. To express this in the form of \eqref{eq:bap}, we take $M_i\subseteq\Hilbert=\mathbb{R}^m$ to be the hyperplane given by
\begin{equation}\label{eq:hyperplane}
 M_i := \{x\in\mathbb{R}^m:\langle a_i,x\rangle = b_i\}\quad\forall i\in\{1,\dots,n\}, 
\end{equation}
where $a_i$ denotes the $i$th row of the matrix $A$.  

To generate random feasible instances of this problem, we generate the matrix $A$ and a feasible point $x^*$, and then compute $b$ according to $b:=Ax^*$. In our experiment, the entries of $A$ were chosen by sampling the standard normal distribution. For each instance, ten replications of each algorithm were performed, starting from different random points $x_0\in\Hilbert$ having $\|x_0\|=10$. The average number of iterations and time required to trigger the termination criteria 
\begin{equation*}
\|x_k-x_{k-1}\|<\epsilon\text{~~with~~}\epsilon=10^{-6}
\end{equation*}
are reported in Table~\ref{t:results}, together with the residual of the final iterate $x_{k'}$ given by $\|Ax_{k'}-b\|$.

This table suggests that the accelerated algorithm performs better than its unaccelerated counterpart in terms of number of iterations (as predicted by our theory) and residual. Since each iteration of the accelerated method requires additional computational work to compute the step size, it is not guaranteed to better in terms of time as the reduction in the number iterations could be offset by the increase in per iteration work. For this problem, this is not the case. Indeed, our results show a significant improvement in terms of time (except for $m=500$ where both algorithms took a similar amount of time), with improvement more pronounced as problem size increases.

\begin{table}
\caption{\rmfamily Results for the best approximation problem subject to ${Ax=b}$ using \eqref{eq:hyperplane} for $A\in\mathbb{R}^{n\times m}$, $b\in\mathbb{R}^m$ and $2n=m$.}\label{t:results}
\rmfamily\centering
\resizebox{\columnwidth}{!}{
\begin{tabular}{lrrrr}
\toprule
 Algorithm  &  $m$ & Iterations & Residual & Time (s)  \\
\midrule
CP       &    500  &       71.6 &  $0.40\times 10^{-4}$  &    0.05 \\
         &  5 000  &       71.4 &  $1.15\times 10^{-4}$  &    1.16 \\
         & 50 000  &       76.0 &  $4.00\times 10^{-4}$  &  137.54 \\ \midrule
Accel CP &    500  &       46.0 &  $0.11\times 10^{-4}$  &    0.07 \\
         &  5 000  &       49.7 &  $0.33\times 10^{-4}$  &    0.81 \\
         & 50 000  &       53.0 &  $1.09\times 10^{-4}$  &   94.20 \\
\bottomrule
\end{tabular}%
}
\end{table}

\section*{Acknowledgements}
This work is supported in part by DE200100063 from the Australian Research Council.
The author would like to thank Janosch Rieger for discussions relating to \cite{R2020} which initiated this work, and the anonymous referee and editor for their helpful comments and suggestions.

\end{document}